\documentclass[runningheads]{llncs}

\usepackage[english]{babel}
\usepackage[utf8]{inputenc}
\usepackage{amsmath}
\usepackage{amssymb}
\usepackage{color}
\usepackage{soul}
\usepackage{graphicx}
\usepackage{cite}

\newcommand{\defeq}{:=}

\newcommand{\norm}[1]{\|#1\|}

\DeclareMathOperator{\BV}{BV}
\DeclareMathOperator{\TV}{TV}
\DeclareMathOperator{\TGV}{TGV}
\DeclareMathOperator{\TVL}{TVL}
\DeclareMathOperator{\TVLp}{\TVL^p}
\newcommand{\TVpwL}{\TV_{pwL}}
\DeclareMathOperator{\divergence}{div}
\let\div\relax
\DeclareMathOperator{\div}{\divergence}

\newcommand{\R}{\mathbb R}

\newcommand{\Mg}{\mathfrak{M}}
\newcommand{\M}{\mathcal{M}}

\renewcommand{\leq}{\leqslant}
\renewcommand{\geq}{\geqslant}
\renewcommand{\phi}{\varphi}

\newcommand{\reg}{\mathcal J}

\newcommand{\C}{\mathcal{C}}
\renewcommand{\L}{\mathcal{L}}

\providecommand{\keywords}[1]{\textbf{Keywords: } #1}

\graphicspath{{../pic/}}

\renewcommand{\hl}[1]{#1}

\begin{document}
\title{A total variation based regularizer promoting piecewise-Lipschitz reconstructions}
\titlerunning{A $\TV$-based regularizer promoting piecewise-Lipschitz reconstructions}

\author{Martin Burger$^1$, Yury Korolev$^2$, Carola-Bibiane Sch\"onlieb$^2$ and Christiane Stollenwerk$^3$}

%\author{First Author\inst{1}\orcidID{0000-1111-2222-3333} \and
%Second Author\inst{2,3}\orcidID{1111-2222-3333-4444} \and
%Third Author\inst{3}\orcidID{2222--3333-4444-5555}}

%%
\authorrunning{M. Burger et al.}

\institute{Department Mathematik, University of Erlangen-N\"urnberg, Cauerstr. 11, 91058  Erlangen, Germany. \email{martin.burger@fau.de} \and
Department of Applied Mathematics and Theoretical Physics, University of Cambridge, Wilberforce Road, Cambridge CB3 0WA, UK. \email{\{yk362,cbs31\}@cam.ac.uk}
\and Institute for Analysis and Numerics, University of M\"unster, Einsteinstr. 62, 48149 M\"unster, Germany. \email{ChristianeStollenwerk@web.de}}

\maketitle              % typeset the header of the contribution
\begin{abstract}
We introduce a new regularizer in the total variation family that promotes reconstructions with a given Lipschitz constant (which can also vary spatially). We prove regularizing properties of this functional and investigate its connections to total variation and infimal convolution type regularizers $\TVLp$ and, in particular, establish topological equivalence. Our numerical experiments show that the proposed regularizer can achieve similar performance as total generalized variation while having the advantage of a very intuitive interpretation of its free parameter, which is just a local estimate of the norm of the gradient. It also provides a natural approach to spatially adaptive regularization.

\end{abstract}
%
%
%%%%%%%%%%%%%%%%%%%%%%%%%%%%%%%
\section{Introduction}
Since it has been introduced in~\cite{ROF}, total variation ($\TV$) has been popular in image processing due to its ability to preserve edges while imposing sufficient regularity on the reconstructions. There have been numerous works studying the  geometric structure of $\TV$-based reconstructions (e.g.,~\cite{Ring:2000, Jalalzai:2016, Peyre:2017, iglesias_mercier_scherzer:2018, MB_YK_JL:2018}). A typical characteristic of these reconstructions is the so-called \emph{staircasing}~\cite{Ring:2000, Jalalzai:2016}, which refers to piecewse-constant reconstructions with jumps that are not present in the ground truth image.
To overcome the issue of staircasing, many other $\TV$-type regularizers have been proposed, perhaps the most successful of which being the Total Generalized Variation ($\TGV$)~\cite{bredies2009tgv, sampta2011tgv}. $\TGV$ uses derivatives of higher order and favours reconstructions that are piecewise-polynomial; in the most common case of $\TGV^2$ these are piecewise-affine. 

While $\TGV$ greatly improves the reconstruction quality compared to $\TV$, the fact that it uses second order derivatives \hl{typically results in slower convergence of iterative optimization algorithms and therefore increases computational costs of the reconstruction}. Therefore, there has been an effort to achieve a performance similar to that of $\TGV$ with a first-order method (i.e. a method that only uses derivatives of the first order). In~\cite{Burger_TVLp_2016, Burger_TVLp_pt2_2016}, infimal convolution type regularizers $\TVLp$ have been introduced that use an infimal convolution of the Radon norm and an $L^p$ norm applied to the weak gradient  of the image. For an $u\in L^1(\Omega)$, $\TVLp$ is defined as follows
\begin{equation*}
\TVLp_{\alpha,\beta}(u) \defeq \min_{w\in L^p(\Omega;\R^d)}\alpha||Du-w||_{\Mg}+\beta||w||_{L^p(\Omega;\R^d)}.
\end{equation*}
\noindent where $D$ is the weak gradient, $\alpha,\beta>0$ are constants and $1<p \leq \infty$. It was shown that for $p=2$ the reconstructions are piecewise-smooth, while for $p=\infty$ they somewhat resemble those obtained with $\TGV$. 

The regularizer we introduce in the current paper also aims at achieving a similar performance with second-order methods while only relying on first order derivatives. It can be seen either as a relaxiation of $\TV$ obtained by extending its kernel from constants to all functions with a given Lipschitz constant (for this reason, we call this new regularizer $\TVpwL$, with `$pwL$' standing for `piecewise-Lipschitz'), or as an infimal convolution type regularizer, where the Radon norm is convolved with the characteristic function of a certain convex set.

We start with the following motivation. Let $\Omega \subset \R^2$ be a bounded Lipschitz domain and $f \in L^2(\Omega)$ a noisy image. Recall the ROF~\cite{ROF} denoising model \looseness=-1
\begin{eqnarray}\nonumber
\min_{u \in BV(\Omega)} \frac{1}{2}\norm{u-f}^2_{L^2(\Omega)}+\alpha \norm{Du}_{\Mg},
\end{eqnarray}
\noindent where $D \colon L^1(\Omega) \to \Mg(\Omega,\R^2)$ is the weak gradient, $\Mg(\Omega,\R^2)$ is the space of vector-valued Radon measures and $\alpha > 0$ is the regularization parameter. 
Introducing an auxiliary variable $g\in\Mg(\Omega,\R^2)$, we can rewrite this problem as follows \looseness=-1
\begin{eqnarray}
\min_{\substack{u \in BV(\Omega)\\g\in\Mg(\Omega,\R^2)}}~\frac{1}{2}\norm{u-f}^2_{L^2(\Omega)}+\alpha \norm{g}_{\Mg}\qquad s.t.~Du=g.\nonumber
\end{eqnarray}
Our idea is to relax the constraint on $Du$ as follows
\begin{eqnarray}
\min_{\substack{u \in BV(\Omega)\\g\in\Mg(\Omega,\R^2)}}~\frac{1}{2}\norm{u-f}^2_{L^2(\Omega)}+\alpha \norm{g}_{\Mg}\qquad s.t.~|Du-g|\leq\gamma\nonumber
\end{eqnarray}
\noindent for some positive constant, function or measure $\gamma$. Here $|Du - g|$ is the variation measure corresponding to $Du-g$ and the symbol $"\leq"$ denotes a partial order in the space of signed (scalar valued) measures $\M(\Omega)$. This problem is equivalent to \looseness=-1
\begin{eqnarray}\label{eq:tvpwl_inspo}
\min_{\substack{u \in BV(\Omega)\\g\in\Mg(\Omega,\R^2)}}~\frac{1}{2}\norm{u-f}^2_{L^2(\Omega)}+\alpha \norm{Du-g}_{\Mg}\quad s.t.~|g|\leq\gamma,
\end{eqnarray}
which we take as the starting point of our approach.

This paper is organized as follows.  
In Section~\ref{sec:primal_and_dual} we introduce the primal and dual formulations of the $\TVpwL$ functional and prove their equivalence. 
In Section~\ref{sec:prop} we prove some basic properties of $\TVpwL$ and study its relationship with other $\TV$-type regularizers. Section~\ref{sec:numerics} contains numerical experiments with the proposed regularizer. \looseness=-1

%%%%%%%%%%%%%%%%%%%%%%%%%%%%%%%%
%\section{Background}\label{sec:background}
%In this section we will briefly discuss the order structure of the space of finite signed measures $\M(\Omega)$.
%
%\begin{definition}\label{def:tvpwl_ordvecsp}
%We call a  measure $\mu \in \M(\Omega)$ positive if for every subset $E\subseteq \Omega$ one has $\mu(E)\geq 0$. For two signed measures $\mu_1,\mu_2 \in \M(\Omega)$ we say that $\mu_1\leq\mu_2$ if $\mu_2-\mu_1$ is a positive measure.
%%Then the space of finite signed measures $\M(\Omega)$ endowed with the order relation $\leq$ is an ordered vector space.
%\end{definition}
%
%To define these operations in the space of signed measures, we will need the following result \cite[Th. III.4.10]{DuSc}.
%
%\begin{theorem}[Hahn decomposition]\label{th:tvpwl_hahn_decomposition}
%Let $(\Omega,\mathcal{F})$ be a measurable space. To every signed measure $\mu$ correspond two $\mathcal{F}$-measurable sets $P$ and $N$ such that:
%\begin{enumerate}
%\item $P\cup N =\Omega$ and $P\cap N=\emptyset$,
%\item $\mu$ is non-negative on measurable subsets of $P$, i.e. $\forall E\in\mathcal{F}$ such that $E\subseteq P$ one has $\mu(E)\geq 0$,
%\item $\mu$ is non-positive on measurable subsets of $N$, i.e. $\forall E\in\mathcal{F}$ such that $E\subseteq N$ one has $\mu(E)\leq 0$.
%\end{enumerate}
%\end{theorem}

%%%%%%%%%%%%%%%%%%%%%%%%%%%%%%%
\section{Primal and Dual Formulations}\label{sec:primal_and_dual}
Let us first clarify the notation of the inequality for signed measures in~\eqref{eq:tvpwl_inspo}.
\begin{definition}\label{def:tvpwl_ordvecsp}
We call a  measure $\mu \in \M(\Omega)$ positive if for every subset $E\subseteq \Omega$ one has $\mu(E)\geq 0$. For two signed measures $\mu_1,\mu_2 \in \M(\Omega)$ we say that $\mu_1\leq\mu_2$ if $\mu_2-\mu_1$ is a positive measure.
\end{definition}

%It can be shown that this partial order is a lattice and defines the absolute value of a measure 
%%can be used to define a positive and a negative parts of a measure as well as its absolute value, 
%which is the total variation~\cite{DS1}. 
%It is also evident that 
%the Radon norm is monotone with respect to this absolute value and 
%It can be shown that with this partial order the space of signed measures $\M(\Omega)$ becomes a Banach lattice.

Now let us  formally define the new regularizer.
\begin{definition}\label{def:tvpwl_primal}
Let $\Omega\subset\R^2$ be a bounded Lipschitz domain, $\gamma\in\M(\Omega)$ be a finite positive measure. For any $u\in L^1(\Omega)$ we define
\begin{eqnarray}
\TVpwL^\gamma(u):=\inf_{g\in\Mg(\Omega,\R^2)}\norm{Du-g}_{\Mg} 
\quad \text{s.t. $|g|\leq\gamma$}\nonumber
\end{eqnarray}
where $||\cdot||_{\Mg}$ denotes the Radon norm and $|g|$ is the variation  measure~\cite{Bredies_Lorenz} corresponding to $g$, i.e. for any subset $E \subset \Omega$
\begin{equation*}
|g|(E) \defeq \sup\left\{\sum_{i=1}^\infty \norm{g(E_i)}_2 \mid E = \bigcup_{i \in \mathbb N} E_i, \,\, \text{$E_i$ pairwise disjoint}  \right\}
\end{equation*}
\noindent  (see also the polar decomposition of measures~\cite{Ambrosio}).
\end{definition}

\hl{The $\inf$ in Definition{~\ref{def:tvpwl_primal}} can actually be replaced by a $\min$ since we are dealing with a metric projection onto a closed convex set in the dual of a separable Banach space. We also note that for $\gamma=0$ we immediately recover total variation.}%, i.e. $\TVpwL^0 = \TV$.}

We defined the regularizer $\TVpwL^\gamma$ in the general case of $\gamma$ being a measure, but we can also choose $\gamma$ to be a Lebesgue-measurable function or a constant. In this case the inequality is understood in the sense $|g| \leq\gamma~d\L$ where $\L$ is the Lebesgue measure, resulting in $|g|$ being absolutely continuous with respect to $\L$. We will not distinguish between these cases in what follows and just write $|g| \leq\gamma$. \looseness=-1

As with standard $\TV$, there is also an equivalent dual formulation of $\TVpwL$.
\vspace{-5mm}
\begin{theorem}\label{thm:tvpwl_dual}
Let $\gamma\in \M(\Omega)$ be a positive finite measure and $\Omega$ a bounded Lipschitz domain. Then for any $u\in L^1(\Omega)$ the $\TVpwL^{\gamma}$ functional can be equivalently expressed as follows
\begin{eqnarray}
\TVpwL^{\gamma}(u)~=~\sup_{\substack{\phi\in \C_0^{\infty}(\Omega;\R^2) \\ |\phi|_2 \leq 1}}~\left\{ \int_{\Omega} u~\div~\phi~dx-\int_{\Omega} |\phi|_2 d\gamma \right\},\nonumber
\end{eqnarray}
\noindent where $|\phi|_2$ denotes the pointwise $2$-norm of $\phi$.
\end{theorem}
\begin{proof}
Since by the Riesz-Markov-Kakutani representation theorem the space of vector valued Radon measures $\Mg(\Omega,\R^2)$ is the dual of the space $\C_0(\Omega,\R^2)$, we 
 can rewrite the expression in Definition~\ref{def:tvpwl_primal} as follows
\begin{eqnarray}
\TVpwL^{\gamma}(u)=\inf_{\substack{g\in\Mg(\Omega,\R^2)\\|g|\leq\gamma}}\norm{Du-g}_{\Mg} 
=\inf_{\substack{g\in\Mg(\Omega,\R^2)\\|g|\leq\gamma }}\sup_{\substack{\phi\in\C_0(\Omega;\R^2)\\|\phi|_2 \leq 1}}(Du-g,\phi).\nonumber
\end{eqnarray}
%In our case the vectorial distribution $Du$ corresponds to the gradient of $u$ and the Hermetian adjoint $D^*$ to the divergence. The functional $(g,\phi)\mapsto(Du-g,\phi)$ is linear in both components. Furthermore, the sets $\{g\in\Mg(\Omega,\R^2)~|~|g|\leq\gamma\}$ and $\{\phi\in\C_0(\Omega)~|~|\phi|_q\leq 1\}$ are compact and convex. Therefore we can use Sion's minimax Theorem from \cite[Th. 3.4]{Si} and exchance infimum and supremum.
%We conduct the following calculation:

In order to exchange $\inf$ and $\sup$, we need to apply a minimax theorem. In our setting  we can use the Nonsymmetrical Minimax Theorem from \cite[Th. 3.6.4]{Borwein_Zhu}.
 Since the set $\{g~|~|g|\leq\gamma\}\subset\Mg(\Omega,\R^2)=(\C_0(\Omega,\R^2))^*$ is bounded,  
convex and closed %(order intervals in a Banach lattice are closed~\cite{Abramovich}) 
and the set \pagebreak
$\{\phi~|~\norm{\phi}_{2,\infty} \leq 1\}\subset\C_0(\Omega,\R^2)$ is convex, we can  swap the infimum and the  supremum and obtain the following representation
\begin{eqnarray*}
& \TVpwL^{\gamma}(u) = \sup\limits_{\substack{\phi\in\C_0(\Omega;\R^2)\\|\phi|_2 \leq 1}}\inf\limits_{\substack{g\in\Mg(\Omega,\R^2)\\|g|\leq\gamma }}( Du-g,\phi)\nonumber\\
& = \sup\limits_{\substack{\phi\in\C_0(\Omega;\R^2)\\|\phi|_2 \leq 1}}[( Du,\phi)-\sup\limits_{\substack{g\in\Mg(\Omega,\R^2)\\|g|\leq\gamma }}(g,\phi)]
 = \sup\limits_{\substack{\phi\in\C_0(\Omega;\R^2)\\|\phi|_2 \leq 1}} [(Du,\phi)-(\gamma ,|\phi|_2)].\nonumber\\
\end{eqnarray*}
\hl{Noting that the supremum can actually be taken over $\phi\in\C_0^\infty(\Omega;\R^2)$, we obtain}
\begin{equation*}
 \TVpwL^{\gamma}(u) = \sup\limits_{\substack{\phi\in\C_0^\infty(\Omega;\R^2)\\|\phi|_2 \leq 1}}[(u, -\div \phi) - (\gamma,|\phi|_2)]
\end{equation*}
\noindent which yields the assertion upon replacing $\phi$ with $-\phi$.
\end{proof}

%%%%%%%%%%%%%%%%%%%%%%%%%%%%%%%
\section{Basic Properties and Relationship with other $\TV$-type regularizers}\label{sec:prop}

%%%
\paragraph*{Influence of $\gamma$.} %First of all let us see what effect the parameter $\gamma$ has on the regularization functional. 
%Choose two signed measures $\gamma,\bar{\gamma}\in \M(\Omega)$ such that $0\leq\gamma\leq\bar{\gamma}$. 
It is evident from Definition~{\ref{def:tvpwl_primal}} that a larger $\gamma$ yields a larger feasible set and a smaller value of $\TVpwL$. Therefore, $\TVpwL^{\gamma}\geq \TVpwL^{\bar{\gamma}}$ whenever $0\leq\gamma\leq\bar{\gamma}$. In particular, we get that $\TVpwL^{\gamma} \leq \TVpwL^0 = \TV$ for any $\gamma \geq 0$. 

%Choose two signed measures $\gamma,\bar{\gamma}\in \M(\Omega)$ such that
%$0\leq\gamma\leq\bar{\gamma}$. Since $(\gamma,|\phi|_2) \leq (\bar \gamma,|\phi|_2)$ for any feasible $\phi$, we conclude, using the dual representation in Theorem~\ref{thm:tvpwl_dual}, that
%\begin{eqnarray}\label{eq:tvpwl_dualestimate1}
%\sup_{\substack{\phi\in\C_0^{\infty}(\Omega;\R^2)\\|\phi|_2 \leq 1}}[(u,\div \phi)-(\gamma,|\phi|_2)]\geq \sup_{\substack{\phi\in\C_0^{\infty}(\Omega;\R^2)\\|\phi|_2 \leq 1}}[(u,\div \phi)-(\bar{\gamma},|\phi|_2)]
%\end{eqnarray}
%\noindent and $\TVpwL^{\gamma}(u)\geq \TVpwL^{\bar{\gamma}}(u)$ for every $u\in L^1(\Omega)$. Since, as noted earlier, $\TVpwL^0 = \TV$, we get that $\TVpwL^{\gamma}(u) \leq \TV(u)$ for every $u\in L^1(\Omega)$ and $\gamma \geq 0$.

%%%
%\noindent\emph{.} From Definition~\ref{def:tvpwl_primal} we see that $\TVpwL$ is the distance from a certain convex set, which is absolute one-homo\-geneous if and only if this set is a singleton consisting of the zero element, i.e. when $\gamma = 0$ and $\TVpwL = \TV$.
%\vspace{0.1cm}

%%%
\vspace{-3mm}
\paragraph*{Lower-semicontinuity and convexity.} Lower-semicontinuity is clear from Definition~{\ref{def:tvpwl_primal}} if we recall that the infimum is actually a minimum. Convexity follows from the fact that $\TVpwL$ is an infimal convolution of two convex functions. 

\vspace{-3mm}
\paragraph*{Absolute one-homogeneity.} Noting that $\TVpwL$ is the distance from the convex set $\{g\in\Mg(\Omega;\R^2)~|~|g| \leq\gamma\}$, we conclude that it is absolute one-homogeneous if and only if this set consists of just zero, i.e. when $\gamma = 0$ and $\TVpwL = \TV$. \looseness=-1 

%Lower-semicontinuity is clear from Definition~\ref{def:tvpwl_primal} if we recall that the infimum is actually a minimum. We also see that $\TVpwL$ is the distance from a certain convex set, which is absolute one-homo\-geneous if and only if this set is a singleton consisting of the zero element, i.e. when $\gamma = 0$ and $\TVpwL = \TV$.
%
%
%Convexity follows from the fact $\TVpwL$ is the infimal convolution of the Radon norm and the characteristic function of the convex set $\{g\in\Mg(\Omega;\R^2)~|~|g| \leq\gamma\}$.
%From the dual representation in Theorem~\ref{thm:tvpwl_dual} it is evident that $\TVpwL^{\gamma}$ is strongly lower-semicontinuous in $L^1(\Omega)$ as a pointwise supremum of continuous functions.

%%%
%\noindent\emph{Convexity.} From the primal formulation (Definition~\ref{def:tvpwl_primal}) we see that $\TVpwL$ is the infimal convolution of the Radon norm and the characteristic function of the convex set $\{g\in\Mg(\Omega;\R^2)~|~|g| \leq\gamma\}$ and so is convex.
%\vspace{0.1cm}

%%%
\vspace{-3mm}
\paragraph*{Coercivity.} We have seen that $\TVpwL^\gamma \leq \TV$ for any $\gamma \geq 0$, i.e. $\TVpwL$ is a lower bound for $\TV$. If $\gamma(\Omega)$ is finite, the converse inequality (up to a constant) also holds and we obtain topological equivalence of $\TVpwL$ and $\TV$. 
\begin{theorem}\label{th:tvpwl_topequiv}
Let $\Omega\subset \mathbb{R}^2$ be a bounded Lipschitz domain and $\gamma\in \M(\Omega)$  a positive finite measure. For every $u\in L^1(\Omega)$ we obtain the following relation:
\begin{eqnarray}
TV(u)-\gamma(\Omega)\leq \TVpwL^{\gamma}(u)\leq TV(u).\nonumber
\end{eqnarray}
\end{theorem}
\begin{proof}
We already established the right inequality. For the left one we observe that for any $g\in\Mg(\Omega,\R^2)$ such that $|g|\leq\gamma$ the following estimate  holds
\begin{equation*}
\norm{Du-g}_{\Mg} \geq \norm{Du}_{\Mg} - \norm{g}_{\Mg} \geq \norm{Du}_{\Mg} - \norm{\gamma}_{\Mg} = \TV(u) - \gamma(\Omega),
\end{equation*}
\noindent which also holds for the infimum over $g$.
\end{proof}

\hl{The left inequality in Theorem~{\ref{th:tvpwl_topequiv}} ensures that $\TVpwL$ is coercive on $\BV_0$, since $\TV(u_n) \to \infty$ implies $\TVpwL^\gamma(u_n) \geq \TV(u_n) - \gamma(\Omega) \to \infty$. Upon adding  the $L^1$ norm, we also get coercivity on $\BV$. This ensures that $\TVpwL$ can be used for regularisation of inverse problems in the same scenarios as $\TV$.}

\hl{Topological equivalence between $\TVpwL$ and $\TV$ is understood in the sense that if one is bounded then the other one is too. Being not absolute one-homogeneous, however, $\TVpwL$ cannot be an equivalent norm on $\BV_0$.}

%%%
\vspace{-3mm}
\paragraph*{Null space.}
We will study the null space in the case when $\gamma \in L^\infty_+(\Omega)$ is a Lebesgue measurable function and the inequality $|g| \leq \gamma$ in Definition~\ref{def:tvpwl_primal} is understood in the sense that $|g| \leq \gamma d\L$ with $\L$ being the Lebesgue measure.

\begin{proposition}\label{prop:tvpwl_nullspace}
Let $u\in L^1(\Omega)$ and $\gamma \geq 0$ be an $L^\infty$ function. Then $\TVpwL^{\gamma}(u)=0$ if and only if the weak derivative $Du$ is absolutely continuous with respect to the Lebesgue measure and its $2$-norm is bounded by $\gamma$ a.e.
\end{proposition}
\begin{proof}
We already noted that, since the space $\Mg(\Omega;\R^2)$ is the dual of the separable Banach space $C_0(\Omega,\R^2)$, the infimum in Definition~\ref{def:tvpwl_primal} is actually a minimum and there exists a $\tilde{g}\in \Mg(\Omega;\R^2)$ with $|g| \leq \gamma d\L$ such that 
\begin{equation*}
\TVpwL^{\gamma}(u)=\norm{Du-\tilde{g}}_{\Mg}.
\end{equation*}
If $\TVpwL^{\gamma}(u)=0$,  then $Du=\tilde{g}$ and therefore $0 \leq |Du|\leq \gamma~d\L$, which implies that $|Du|$ is absolutely continuous with respect to $\L$ and can be written as 
\begin{eqnarray}
|Du|(A)=\int_{A}f~d\L\nonumber
\end{eqnarray}
for any $A\subseteq \Omega$. The function $f\in L^1_+(\Omega)$ is the Radon-Nikodym derivative $\frac{d|Du|}{d\L}$. 
From the condition $|Du|\leq\gamma~d\L$ we immediately get that $f \leq \gamma$ a.e.
\end{proof}

\begin{remark}
We notice that the set $\{u \in L^1(\Omega) \colon \TVpwL(u) = 0\}$ is not a linear subspace, therefore, we should rather speak of the null set than the null space.
\end{remark}

\begin{remark}
Proposition~\ref{prop:tvpwl_nullspace} implies that all functions in the null set are Lipschitz continuous with (perhaps, spatially varying) Lipschitz constant $\gamma$; hence the name of the regularizer.
\end{remark}

%%%
\vspace{-3mm}
\paragraph*{Luxemburg norms.}  \hl{For a positive convex nondecreasing function $\phi \colon \R_+ \to \R_+$ with $\phi(0)=0$ the Luxemburg norm $\norm{\cdot}_\phi$ is defined as follows} \cite{Luxemburg:1955}
\begin{equation*}
\norm{u}_\phi = \sup\left\{\lambda > 0 \colon \int \phi(|u| / \lambda)\, d\mu \leq 1 \right\}.
\end{equation*}
 \hl{We point out a possible connection between $\TVpwL$ and a Luxemburg norm corresponding to $\phi(x) = (x-c)_+$ with a suitable constant $c>0$. However, we do not investigate this connection in this paper.}\looseness=-1

%%%
%\subsection{Relationship with Total Variation}\label{sec:tv}

%From the dual formulation of $\TVpwL^{\gamma}$ it is evident that with $\gamma=0$ we obtain the total variation, i.e.:
%\begin{eqnarray}\label{eq:tvpwl_dualestimate2}
%\TVpwL^0(u) = \sup_{\substack{\phi\in\C_0^{\infty}(\Omega;\R^2)\\ |\phi|_2 \leq 1}}(u,\div\phi) = \TV(u).
%\end{eqnarray}
%\noindent This is also directly evident from the primal formulation in Definition~\ref{def:tvpwl_primal}.
%
%From~\eqref{eq:tvpwl_dualestimate1} we get that $\TVpwL^\gamma(u) \leq \TVpwL^0(u) = \TV(u)$ for any $\gamma \geq 0$, i.e. $\TVpwL$ is a lower bound for $\TV$. If $\gamma(\Omega)$ is finite, the converse inequality (up to a constant) also holds and we obtain topological equivalence of $\TVpwL$ and $\TV$.

%%%
%\subsection{Relationship with Infimal Convolution Type regularizers}

%%%

\vspace{-3mm}
\paragraph*{Relationship with Infimal Convolution Type regularizers.} We would like to highlight a relationship to infimal convolution type regularizers  $\TVLp$~\cite{Burger_TVLp_2016, Burger_TVLp_pt2_2016}. Indeed, 
as already noticed earlier, $\TVpwL$ can be written as an infimal convolution
\begin{eqnarray}\label{eq:tvpwl_infconv}
TV^{\gamma}_{pwL}(u)&=&\inf_{\substack{g\in\Mg(\Omega;\R^2)\\|g|\leq\gamma}}\norm{Du-g}_{\M} = \inf_{g\in\Mg(\Omega,\R^2)}\{\norm{Du-g}_{\Mg}+\chi_{C_{\gamma}}(g)\},
\end{eqnarray}
where $C_{\gamma}:=\{\eta\in\Mg(\Omega;\R^2)~|~|\eta|_p\leq\gamma\}$. If $\gamma>0$ is a constant, we obtain a bound on the $2,\infty$-norm of $g$. This highlights a connection to the $\TVL^\infty$ regularizer~\cite{Burger_TVLp_pt2_2016}: for any particular weight in front of the $\infty$-norm in $\TVL^\infty$  (and for a given $u$), 
the auxiliary variable $g$ will have some value of the $\infty$-norm and if we use this value as $\gamma$ in $\TVpwL^\gamma$, we will obtain the same reconstruction. 
The important difference is that with $\TVL^\infty$ we don't have direct control over this value and can only influence it in an indirect way through the weights in the regularizer. \pagebreak
 In $\TVpwL$ this parameter  is given explicitly and can be either  obtained using additional a priori information about the ground truth or estimated from the noisy image. \looseness=-1

Similar arguments can be made in the case of spatially variable $\gamma$ if the weighting in  $\TVL^\infty$ is also allowed to vary spatially.

%%%%%%%%%%%%%%%%%%%%%%%%%%%%%%%
\section{Numerical Experiments}\label{sec:numerics}
In this section we want to compare the performance of the proposed first order regularizer with a second order regularizer, $\TGV$, in image denosing. We consider images (or 1D signals) corrupted by Gaussian noise with a known variance and use the residual method~\cite{engl:1996} to reconstruct the noise-free image, i.e. we solve (in the discrete setting)
\begin{eqnarray}\label{eq:num_tv_modelnoise}
\min_{u \in \R^N}\reg(u) \quad \text{s.t. $\norm{u-f}^2_2 \leq \sigma^2 \cdot N$},
\end{eqnarray}
\noindent where $f$ is the noisy image, $\reg$ is the regularizer ($\TVpwL$ or $\TGV$), $\sigma$ is the standard deviation of the Gaussian noise and $N$ is the number of pixels in the image (or the number of entries in the 1D signal). We solve all problems in  MATLAB using CVX~\cite{cvx}. For $\TGV$ we use the parameter $\beta = 1.25$, which is in the range $[1,1.5]$ recommended in~\cite{DlosR_CS_TV_bilevel:2018}.

A characteristic feature of the proposed regularizer $\TVpwL$ is its ability to efficiently encode the information about the gradient of the ground truth (away from jumps) if such information is available. Our experiments showed that the quality of the reconstruction significantly depends on the quality of the (local) estimate of the norm of the gradient of the ground truth. 

The ideal application for $\TVpwL$ would be one where we have a good estimate of the gradient of the ground truth away from jumps, which, however, may occur at unknown locations and be of unknown magnitude. If such an estimate is not available, we can roughly estimate the gradient of the ground truth from the noisy signal, which is the approach we take.

%%%
\subsection{1D experiments.}

We consider the ground truth shown in Figure~\ref{fig:1D}a (green dashed line). The signal is discretized using $N = 1000$ points.  We add Gaussian noise with variance $\sigma = 0.1$ and obtain the noisy signal shown in the same Figure (blue solid line).

\begin{figure}[t]
\centering
\includegraphics[width=0.9\textwidth]{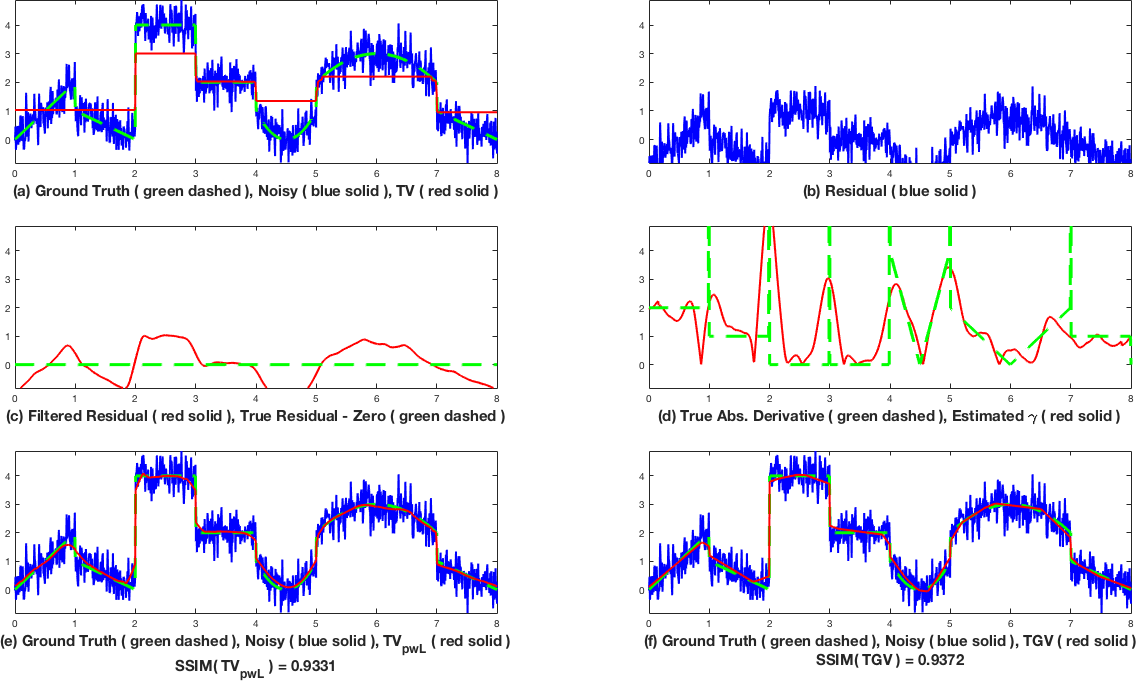}
\caption{The pipeline for the reconstruction with $\TVpwL$ (a-e). An overregularized $\TV$ reconstruction is used to detect and partially eliminate the jumps (a). The residual (b) is filtered (c) and numerically differentiated (d). The absolute value of the obtained derivative is used as the parameter $\gamma$ (d) for $\TVpwL^\gamma$. The reconstruction using $\TVpwL^\gamma$ (e) follows well the structure of the ground truth apart from a small artefact at around $2$. $\TGV$ also yileds a good reconstruction (f), although it tends to approximate the solution with a piecewise-affine function in areas where the ground truth is not affine (e.g., between $4$ and $5$). Both regularizers yield similar (high) values of SSIM.}
\label{fig:1D}
\end{figure}

To use $\TVpwL$, we need to estimate the derivative of the true signal away from the jumps. Therefore, we need to detect the jumps, but leave the signal  intact away from them (up to a constant shift). This is exactly what happens if the image is overregularized with $\TV$. We compute a $\TV$ reconstruction by solving the ROF model 
\begin{eqnarray}\label{eq:num_ROF}
\min_{u \in \R^N}\frac{1}{2}\norm{u-f}^2_2 + \alpha \TV(u)
\end{eqnarray}
\noindent with a large value of $\alpha$ (in this example we took $\alpha = 0.5$). The result is shown in Figure~\ref{fig:1D}a (red solid line). The residual, which we want to use to estimate the derivative of the ground truth, \pagebreak is shown in Figure~\ref{fig:1D}b. Although the jumps have not been removed entirely, this signal can be used to estimate the derivative using filtering. \looseness=-1

The filtered residual (we used the build-in MATLAB function 'smooth' with option 'rlowess' (robust weighted linear least squares) and $\omega = 50$) is shown in 
Figure~\ref{fig:1D}c. This signal is sufficiently smooth to be differentiated. We use central differences; to suppress the remaining noise in the filtered residual we use a step size for differentiation that is $20$ times the original step size. The result is shown in Figure~\ref{fig:1D}d (reg solid line) along with the true derivative (green dashed line). We use the absolute value of the so computed derivative as the parameter $\gamma$.\looseness=-1

The reconstruction obtained using $\TVpwL^\gamma$ is shown in  Figure~\ref{fig:1D}e. We see that the reconstruction is best in areas where our estimate of the true derivative was most faithful (e.g., between $4$ and $5$). But also in other areas the reconstruction is good and preserves the structure of the ground truth rather well. We notice a small artefact at the value of the argument of around $2$; examining the estimate of the derivative in Figure~\ref{fig:1D}d and the residual in Figure~\ref{fig:1D}b, we notice that $\TV$ was not able the remove the jump at this location and therefore the estimate of the derivative was too large. This allowed the reconstruction to get too close to the data at this point.

\hl{We also notice that the jumps are sometimes reduced, with a characteristic linear cut near the jump (e.g., near $x=1;2;3$ and $4$). This can have different reasons.  For the jumps near $x=3$ and $4$ we see that the estimate of $\gamma$ is too large (the true derivative is zero), which allows the regulariser to cut the edges.  For the jumps near $x=1$ and $2$ the situation is different. At these positions a negative slope in the ground truth is followed by a positive jump. Since $\gamma$ only constraints the absolute value of the gradient, even with a correct estimate of $\gamma$ the regulariser will reduce the jump, going with the maximum slope in the direction of the jump. Functions with a negative slope followed by a positive jump are also problematic for $\TV$, since they do not satisfy the source condition (their subdifferential is empty}~\cite{Bungert_Burger_Chambolle_Novaga_19}). \hl{In such cases $\TV$ will also always reduce the jump. }  \looseness=-1

Figure~\ref{fig:1D}e shows the reconstruction obtained with $\TGV$. The reconstruction is quite good, although it is often piecewise-affine where the ground truth is not, e.g. between $4$ and $5$ or between $5$ and $7$. As expected, both regularizers tend to push the reconstructions towards their kernels, but, since $\TVpwL^\gamma$ with a good choice of $\gamma$ contains the ground truth in its kernel (up to the jumps), it yields reconstructions that are more faithful to the structure of the ground truth.

%%%
\subsection{2D experiments.}
In this Section we study the performance of $\TVpwL$ in denoising of 2D images. We use two images - ``cameraman'' (Figure~\ref{fig:cameraman}a) and ``owl'' (Figure~\ref{fig:owl}a). Both images  have the resolution $256 \times 256$ pixels and  values in the interval $[0,255]$. The images are corrupted with Gaussian noise with standard deviation $\sigma = 0.1 \cdot 255 = 25.5$ (Figures~\ref{fig:cameraman}b and~\ref{fig:owl}b).

\begin{figure}[t]
\centering
\includegraphics[width=0.95\textwidth]{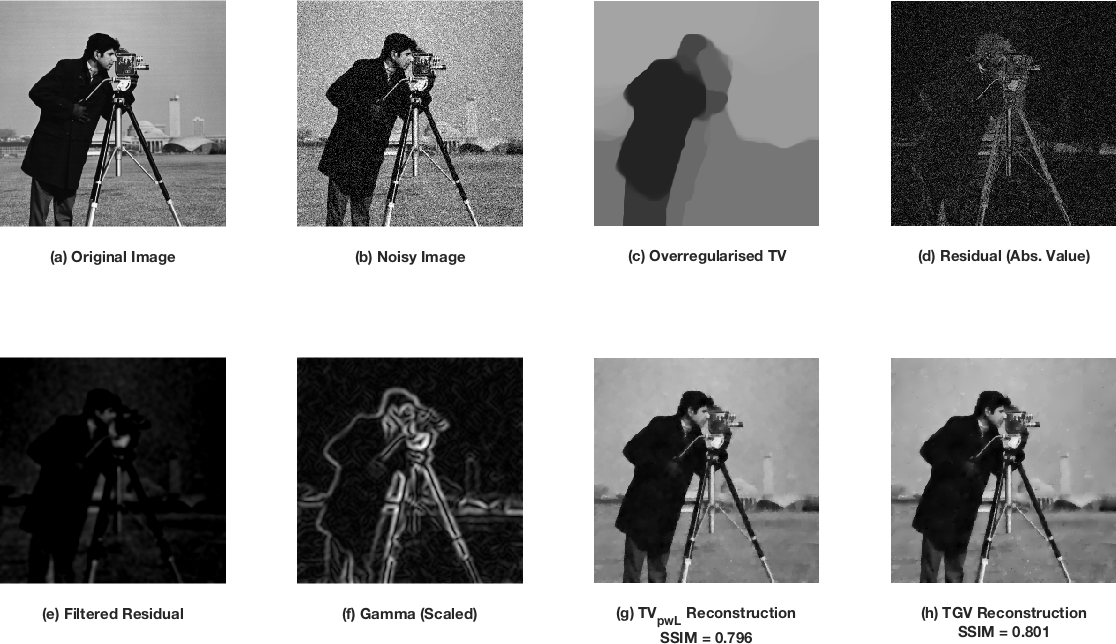}
\caption{The pipeline for the reconstruction of the ``cameraman'' image with $\TVpwL$ (c-g). An overregularized $\TV$ reconstruction is used to detect and partially eliminate the jumps (c). (Note that $\TV$ managed to segment the picture into piecewise-constant regions rather well.) The residual (d) is filtered (e) and numerically differentiated. The norm of the obtained gradient is used as the parameter $\gamma$ for $\TVpwL^\gamma$ ((f), $\gamma$ is scaled to the interval $[0,255]$ for presentation purposes). The reconstructions using $\TVpwL^\gamma$ (g) and $\TGV$ (h) are almost identical. Both preserve edges and are rather smooth away from them. Details are rather well preserved (see, e.g., the pillars of the building in the background as well as the face of the cameraman; the texture of the grass is lost in both cases, however). Relatively homogeneous regions in the original image are also relatively homogeneous in the reconstruction, yet they are not piecewise constant. SSIM values differ very little.}
\label{fig:cameraman}
\end{figure}

\begin{figure}[t!]
\centering
\includegraphics[width=0.95\textwidth]{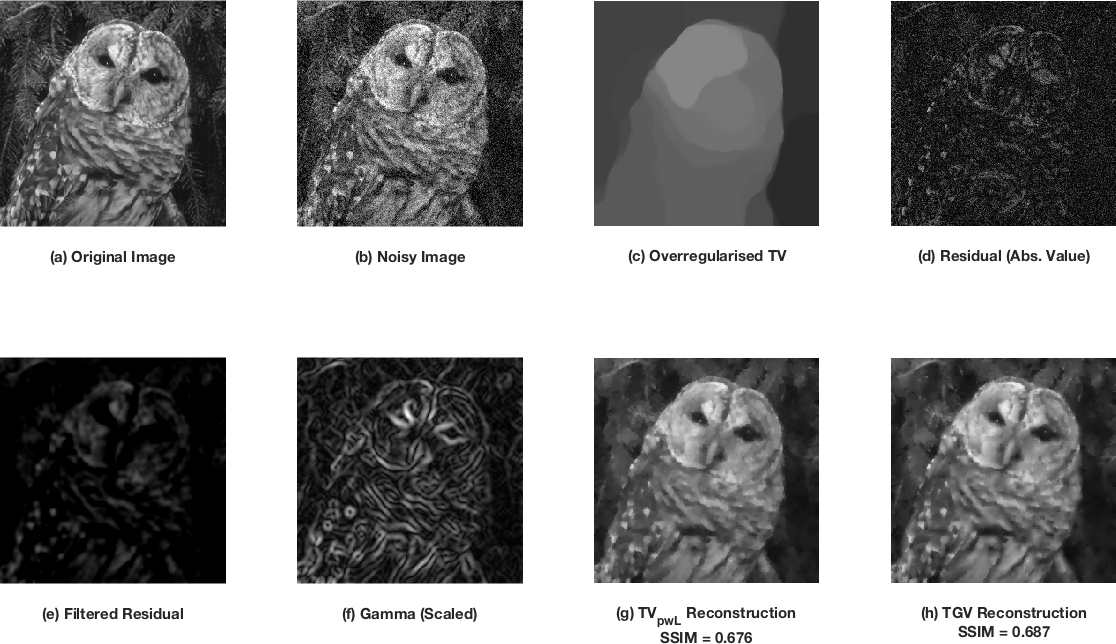}
\caption{The pipeline for the reconstruction of the ``owl'' image with $\TVpwL$ (c-g). An overregularized $\TV$ reconstruction is used to detect and partially eliminate the jumps (c). (This time the segmentation obtained by $\TV$ is not perfect -- perhaps too detailed -- but still rather good.) The residual (d) is filtered (e) and numerically differentiated. The norm of the obtained gradient is used as the parameter $\gamma$ for $\TVpwL^\gamma$ ((f), $\gamma$ is scaled to the interval $[0,255]$ for presentation purposes). This time the residual is not as clear as in the ``cameraman'' example and the estimated $\gamma$ seems noisier. However, it still mainly follows the structure of the original image. The reconstruction using $\TVpwL^\gamma$ (g) preserves the edges and well reconstructs some details in the image, e.g., the feathers of the owl. Other details, however, are lost (the needles of the pine tree in the background). Looking at $\gamma$ in this region, we notice that it is rather irregular and does not capture the structure of the ground truth.
The $\TGV$ reconstruction (h) is again very similar to $\TVpwL$ and  SSIM values are very close.}
\label{fig:owl}
\end{figure}

The pipeline for the reconstruction using $\TVpwL$ is the same as in 1D. We obtain a piecewise constant image by solving an overregularized ROF problem~\eqref{eq:num_ROF} with $\alpha = 500$ (Figures~\ref{fig:cameraman}c and~\ref{fig:owl}c) and compute the residuals (Figures~\ref{fig:cameraman}d and~\ref{fig:owl}d). Then we smooth the residuals using a Gauss filter with $\sigma = 2$ (Figures~\ref{fig:cameraman}e and~\ref{fig:owl}e) and compute its derivatives in the $x$- and $y$-directions using the same approach as in 1D (central differences with a different step size; we used step size $3$ in this example). These derivatives are used to estimate $\gamma$, which is set equal to the norm of the gradient. Figures~\ref{fig:cameraman}f and~\ref{fig:owl}f show $\gamma$ scaled to the interval $[0,255]$ for better visibility. We use the same parameters (for Gaussian filtering and numerical differentiation) to estimate $\gamma$ in both images. Reconstructions obtained using $\TVpwL^\gamma$ and $\TGV$ are shown in Figures~\ref{fig:cameraman}g-h and~\ref{fig:owl}g-h. As in the 1D example, the parameter $\beta$ for $\TGV$ was set to $1.25$.

Comparing the results for both images, we notice that the residual (as well as its filtered version) captures the details in the ``cameraman'' image much better than in the ``owl'' image. The filtered residual in the ``owl'' image seems to miss some of the structure of the original image and this is reflected in the estimated $\gamma$ (which looks much noisier in the ``owl'' image and, in particular, does not capture the structure of the needles of the pine tree in the upper left corner). This might be due to the segmentation achieved by $\TV$, which seems better in the ``cameraman'' image (the one in the ``owl'' image seems to a bit too detailed). This effect might be mitigated by using a better segmentation technique.

This difference is reflected in the reconstructions.  While in the ``cameraman'' image the details are well preserved (e.g., the face of the cameraman or his camera,  as well as the details of the background), in the ``owl'' image part of them are lost and replaced by rather blurry (if not constant) regions; however, in other regions, such as the feathers of the owl, the details are preserved much better, which can be also seen from the estimated $\gamma$ that is much more regular in this area and closer to the structure of the ground truth.
We also notice some loss of contrast in the $\TVpwL$ reconstruction. Perhaps, it could be dealt with by adopting the concept of debiasing~\cite{Deladelle1, Burger_Rasch_debiasing} in the setting of $\TVpwL$, however, it is not clear yet, what is the structure of the model manifold in this case.
%; the fact that $\TVpwL$ is not absolute one-homogeneous would also require some adaptations. 
%The study of debiasing for $\TVpwL$ is the topic for future research.

The $\TGV$ reconstructions look strikingly similar to those obtained by $\TVpwL$. Structural similarity between these reconstructions (i.e. SSIM computed using on of them as the reference) is $0.98$ for the ``cameraman'' image and $0.97$ for the ``owl'' image. Although the $\TGV$ reconstructions depend on the parameter $\beta$ any may differ more from $\TVpwL$ for other values of $\beta$, the one we chose here ($\beta = 1.25$) is reasonable and lies within the optimal range reported in~\cite{DlosR_CS_TV_bilevel:2018}.\looseness=-1

There are two main messages to be taken from these experiments.  The first one is that $\TVpwL$ is able to almost reproduce the reconstructions obtained  
using $\TGV$ with a reasonable choice of the parameter $\beta$, which is a very good performance for a method that does not use higher-order derivatives. 
The second one is that the performance of $\TVpwL$ greatly depends on the quality of the estimate of $\gamma$. When we are able to well capture the structure of the original image in this estimate, the structure of the reconstructions is rather close to that of the ground truth. To further illustrate this point, we show in Fig.~\ref{fig:exact_gamma} $\TVpwL^\gamma$ reconstructions in the ideal scenario when $\gamma$ is estimated from the ground truth as the local magnitude of the gradient.  The quality of the reconstructions is very good, suggesting that with a better strategy of estimating the gradient $\TVpwL^\gamma$ could achieve even better performance.

\begin{figure}[t]
\centering
\includegraphics[width=0.6\textwidth]{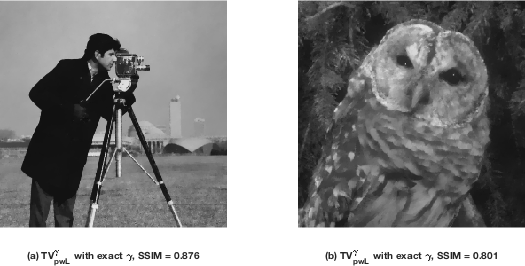}
\caption{In the ideal scenario when $\gamma$ is estimated from the ground truth, $\TVpwL^\gamma$ is able to reproduce the original image almost perfectly}
\label{fig:exact_gamma}
\end{figure}

%%%%%%%%%%%%%%%%%%%%%%%%%%%%%%%%%%%%%%%%%%%%%%%%%%%%%%%%%%%%%%
\section{Conclusions}
We proposed a new $\TV$-type regularizer that can be used to decompose the image into a jump part and a part with Lipschitz continuous gradient (with a given Lipschitz constant that is also allowed to vary spatially). Functions whose gradient does not exceed this constant lie in the kernel of the regularizer and are not penalized. By smartly choosing this bound we can hope to put the ground truth into the kernel (up to the jumps) and thus not penalize any structure that is present in the ground truth.

%Due to the numerical complexity of second order regularization, such as $\TGV$ regularization, there is a need for first order approaches that can yield results of comparable quality. With the proposed regularizer, $\TVpwL$, we aim at achieving exactly this. The kernel of $\TVpwL$ consists of functions with a gradient that does not exceed a certain value that is also allowed to change spatially. By smartly choosing this bound we can hope to put the ground truth into the kernel (up to the jumps) and thus not penalise any structure that is present in the ground truth.\looseness=-3

In this paper we presented, in the context of denoising, an approach to estimating this bound from the noisy image. The approach is based on segmenting the image (and compensating for the jumps) using overregularized $\TV$ and estimating the local bound on the gradient using filtering. 
Our numerical experiments showed that $\TVpwL$ can produce reconstructions that are very similar to $\TGV$, however, the results significantly depend on the quality of the estimation of the local bound on the gradient. Using a more sophisticated estimation technique is expected to further improve the reconstructions. 
The ideal application for $\TVpwL$ would be one where there is some information about the magnitude of the Lipschitz part of the gradient. %We are keen on learning about such problems.

%Using $\TVpwL$ in problems with more complicated forward operators, such as deblurring and tomography, and adapting the strategy of estimating the bound on the gradient in the present paper to these settings is an interesting problem for future work. It would be also interesting to see how $\TVpwL$ performs in inpainting.

%An interesting problem for future work is to explore how the concept of debiasing can be adapted to the framework of $\TVpwL$ and, in particular, what the structure of the model manifold would be.

\section*{Acknowledgments}
This work was supported by the European Union's Horizon 2020 research and innovation programme under the Marie Sklodowska-Curie grant agreement No 777826 (NoMADS).  MB acknowledges further support by ERC via Grant EU FP7 -- ERC Consolidator Grant 615216 LifeInverse.
YK acknowledges  support of the Royal Society through a Newton International Fellowship. YK also acknowledges  support of the Humbold Foundataion through a Humbold Fellowship he held at the University of M\"unster when this work was initiated.
CBS acknowledges support from the Leverhulme Trust project on Breaking the non-convexity barrier, EPSRC grant Nr. EP/M00483X/1, the EPSRC Centre Nr. EP/N014588/1, the RISE projects CHiPS and NoMADS, the Cantab Capital Institute for the Mathematics of Information and the Alan Turing Institute. We gratefully acknowledge the support of NVIDIA Corporation with the donation of a Quadro P6000 and a Titan Xp GPUs used for this research.


\bibliographystyle{splncs04}
\bibliography{abbrevs.bib,BL.bib,IP.bib}

\end{document}